\documentclass[11pt]{amsart}
\usepackage{dynkin-diagrams}
\usepackage[utf8]{inputenc}
\usepackage[style=alphabetic,firstinits,backend=bibtex, uniquename=init,
url=true,hyperref]{biblatex}
\DeclareFieldFormat{labelalpha}{\thefield{entrykey}}
\DeclareFieldFormat{extraalpha}{}
\DeclareFieldFormat{postnote}{#1}
\DeclareFieldFormat{multipostnote}{#1}
\usepackage{tikz}
\usepackage{tikz-cd}
\usetikzlibrary{decorations.text,calc,arrows.meta}
\usepackage{stmaryrd}
\usepackage{array}
\usepackage{nicematrix}
\usepackage{varwidth}
\usepackage{filecontents}
\usepackage[titletoc]{appendix}
\usepackage{mathtools}
\usepackage{amsthm}
\usepackage{amsmath}
\usepackage{mathabx}
\usepackage{amsfonts}
\usepackage{amssymb}
\usepackage{mathrsfs}
\usepackage{multirow}
\usepackage{xypic}
\usepackage{color}
\usepackage{xcolor}
\usepackage{titlesec}
\usepackage{titletoc}
\usepackage[colorlinks]{hyperref}
\definecolor{cclr}{rgb}{25,25,112}
\hypersetup{
citecolor=[rgb]{0.15, 0.15, 0.68},
linkcolor=gray,   
urlcolor=magenta}

\bibliography{heisenberg}

\newtheorem{cor}{Corollary}

\newtheorem{thm}{Theorem}

\newtheorem*{egothm}{{Property EPL}}
\newtheorem*{egfthm}{{Property EL}}
\newtheorem{setup}{{Setup}}
\newtheorem*{egithm}{{Property SSD}}
\def\presuper#1#2%
  {\mathop{}%
   \mathopen{\vphantom{#2}}^{#1}%
   \kern-\scriptspace%
   #2}

\newcommand{\Ad}{\operatorname{Ad}}
\newcommand{\cyc}{\operatorname{cyc}}

\newcommand{\gr}{\operatorname{gr}}

\newcommand{\Ss}{{\operatorname{ss}}}

\newcommand{\wh}[1]{\breve{#1}}
\newcommand{\wt}[1]{\widetilde{#1}}

\newcommand{\Lie}{\operatorname{Lie}}

\newcommand{\Span}{\operatorname{span}}
\newcommand{\GSp}{\operatorname{GSp}}
\newcommand{\GO}{\operatorname{GO}}

\newcommand{\ad}{\operatorname{ad}}

\newcommand{\coker}{\operatorname{coker}}

\newcommand{\Gal}{\operatorname{Gal}}

\newcommand{\red}{{\operatorname{red}}}

\newcommand{\Res}{\operatorname{Res}}

\newcommand{\Int}{{\operatorname{Int}}}
\newcommand{\Herr}{{\operatorname{Herr}}}
\newcommand{\Diag}{{\operatorname{Diag}}}
\newcommand{\Dyn}{{\operatorname{Dyn}}}
\newcommand{\Tor}{{\operatorname{Tor}}}

\newcommand{\GSO}{{\operatorname{GO}}}

\newcommand{\crys}{{\operatorname{crys}}}

\newcommand{\Sp}{\operatorname{Sp}}
\newcommand{\Mat}{\operatorname{Mat}}
\newcommand{\GL}{\operatorname{GL}}

\newcommand{\SO}{\operatorname{SO}}

\newcommand{\Img}{\operatorname{Im}}

\newcommand{\ab}{\operatorname{ab}}

\newcommand{\bG}{\breve{G}}

\newcommand{\Hom}{\operatorname{Hom}}

\newcommand{\spec}{\operatorname{Spec}}
\newcommand{\rank}{\operatorname{rank}}

\titleformat{\section}[runin]{\normalfont\bfseries}{\thesection.}{3pt}{}
\titleformat{\subsection}[runin]{\normalfont\bfseries}{\thesubsection.}{3pt}{}
\titleformat{\subsubsection}[runin]{\normalfont\bfseries}{\thesubsubsection.}{3pt}{}
\titleformat{\paragraph}[runin]{\normalfont\bfseries}{}{3pt}{}
\renewcommand{\thesection}{\arabic{section}}

\titleformat{\section}{\normalfont\large\bfseries}{\thesection.~~}{1em}{}

\newcommand{\fx}{\mathfrak{x}}
\newcommand{\fy}{\mathfrak{y}}

\newcommand{\bA}{\mathbb{A}}

\newcommand{\cX}{\mathcal{X}}

\newcommand{\bFp}{\bar{\mathbb{F}}_p}
\newcommand{\bZp}{\bar{\mathbb{Z}}_p}
\newcommand{\bQp}{\bar{\mathbb{Q}}_p}

\newcommand{\bQ}{\mathbb{Q}}
\newcommand{\bZ}{\mathbb{Z}}

\newcommand{\Qp}{\mathbb{Q}_p}
\newcommand{\lsup}[2]{{^{#1}\!#2}}

\dottedcontents{section}[2.5em]{}{2.9em}{1pc}

\begin{document}

\title{Heisenberg varieties and the existence of de Rham lifts}
\author{Zhongyipan Lin}

\begin{abstract}
We establish the existence of de Rham lifts
of Langlands parameters
(or Galois representations)
for unitary, orthogonal and symplectic
(similitude)
groups of arbitrary rank.
Our results are unconditional 
except for the assumption $p>2$.
\end{abstract}

\maketitle

\tableofcontents

\section{Introduction}~

Let $G$ be a connected reductive group over a $p$-adic
field $F$, $p\neq 2$.
Set $\lsup LG:=\wh G\rtimes \Gal(K/F)$
where $\wh G$ is the complex dual of $G$ (defined over $\bZ_p$)
and $K$ is the splitting field of $G$.
In this paper, we prove that
\begin{thm}
\label{thm:main}
If $\lsup LG$ is one of
\begin{itemize}
\item the $L$-dual of unitary groups $\lsup LU_n$,
\item (the neutral component of) the
orthogonal similitude groups $\GSO_n$, or
\item the symplectic similitude groups $\GSp_{n}$,
\end{itemize}
then all Langlands parameters $\bar\rho:\Gal_F\to \lsup LG(\bFp)$
admit a
potentially diagonalizable,
potentially crystalline lift $\rho:\Gal_F\to \lsup LG(\bZp)$
of regular Hodge type.
\end{thm}

Combining the local-to-global lifting theorem
\cite[Theorem A]{FKP21}
and the local lifting theorem when $l\neq p$
\cite[Theorem 1.1]{BCT24},
we have the following geometric lifting theorem:

\begin{cor}
Let $\bG$ be either $\GSp_{n}$
or $\GSO_n$ over $\bZ$,
and assume $p\gg_G 0$.
Let $\mathbf{F}$ be a totally real number field
and let $\bar\rho:\Gal_{\mathbf F}\to \bG(\bFp)$
be an odd continuous representation
such that $\bar\rho|_{\Gal_{\mathbf{F}(\zeta_p)}}$
is $\bG$-absolutely irreducible.
Then $\bar\rho$ admits
a lift $\bar\rho:\Gal_{\mathbf F}\to \bG(\bZp)$
that is geometric in the sense
of the Fontaine-Mazur conjecture.
\end{cor}

\subsection{Outline of the strategy}~

Let $P\subset \wh G$ be a $\Gal(K/F)$-stable parabolic.
Write $\lsup LP$ for $P\rtimes \Gal(K/F)$.
A mod $p$ Langlands parameter is either elliptic,
or factors through some $\lsup LP$.
In our previous paper \cite{L23A},
we classified elliptic mod $p$ Langlands parameters for $G$
and constructed their de Rham lifts.
In this paper, we shift our attention to parabolic mod $p$ Langlands parameters.

Let $\bar\rho:\Gal_F\to \lsup LP(\bFp)$ be a parabolic mod $p$
Langlands parameter, and 
we seek to address the following question:
\begin{quotation}
Does there exist
a de Rham lift $\rho:\Gal_F\to \lsup LP(\bar\bZ_p)$
of regular Hodge type?
\end{quotation}
The $G=\GL_n$-case is resolved in the book
by Emerton and Gee.
We briefly describe the strategy of \cite{EG23}
and explain the challenge for groups which are not $\GL_n$.
Let $\lsup LM = \wh M\rtimes \Gal(K/F)$ denote a $\Gal(K/F)$-stable Levi subgroup of
$\lsup LP$ and write $U$ for the unipotent radical of $P$.
Write $\bar\rho_M:\Gal_F\xrightarrow{\bar\rho} \lsup LP(\bFp)\to \lsup LM(\bFp)$
for the $\lsup LM$-semisimplification of $\bar\rho$.
The construction $\rho$ follows a $2$-step process:
\begin{itemize}
\item Step 1: Carefully choose a lift $\rho_M:\Gal_F\to \lsup LM(\bar\bZ_p)$ of $\bar\rho_M$.
\item Step 2: 
Endow $U(\bar\bZ_p)$ with the $\Gal_F$-action induced by $\rho_M$.
Show that the image of
$$
H^1(\Gal_F, U(\bar\bZ_p)) \to H^1(\Gal_F, U(\bFp))
$$
contains the cocycle corresponding to $\bar\rho$.
\end{itemize}

\subsection{Partial lifts after abelianization and the work of Emerton-Gee}~

Let $[U, U]$ denote the derived subgroup of $U$
and write $U^{\ab}:=U/[U, U]$
for the abelianization of $U$.

The first approximation of a de Rham lift of $\bar\rho$
is a continuous group homomorphism
$$\Gal_F\to \frac{\lsup LP}{[U,U]}(\bar\bZ_p)$$
lifting $\bar\rho$ modulo $[U,U]$.
The groundbreaking idea presented in \cite{EG23} is that we can achieve the following:
\begin{itemize}
\item Step 1: Choose a lift $\rho_M:\Gal_F\to \lsup LM(\bar\bZ_p)$ of $\bar\rho_M$.
\item Step 2: Guarantee the image of
$$
H^1(\Gal_F, U^{\ab}(\bar\bZ_p)) \to H^1(\Gal_F, U^{\ab}(\bFp))
$$
contains the cocycle corresponding to $\bar\rho \mod [U,U]$,
\end{itemize}
as long as we can estimate the dimension of certain substacks
of the reduced Emerton-Gee stacks for $M$.

We formalize the output of their geometric argument as follows.

\begin{setup}
\label{setup}
Let $\bar\rho_M:\Gal_F\to \lsup LM(\bFp)$
be an $L$-parameter.
Let $\spec R$ be a non-empty potentially crystalline
framed
deformation ring of $\bar\rho_M:\Gal_F
\xrightarrow{\bar\rho_M} \lsup LM(\bFp)$
such that for some $x \in \spec R(\bar\bQ_p)$,
$H^1_{\crys}(\Gal_K, U^{\ab}(\bar\bQ_p))
=H^1(\Gal_K, U^{\ad}(\bar\bQ_p))$.
\end{setup}

\begin{egothm}
\label{axiom:eg}
(Existence of partial lifts)
For any cocycle $[\bar c]\in H^1(\Gal_F, U^{\ab}(\bFp))$,
there exists a point $y\in \spec R(\bar\bZ_p)$ 
equipping $U^{\ab}(\bar\bZ_p)$ with a $\Gal_F$-action,
such that the image of
$$H^1(\Gal_F, U^{\ab}(\bar\bZ_p)) \to H^1(\Gal_F, U^{\ab}(\bFp))$$
contains $[\bar c]$.
\end{egothm}

\begin{egfthm}
\label{axiom:eg}
(Existence of full lifts)
For any cocycle $[\bar c]\in H^1(\Gal_F, U(\bFp))$,
there exists a point $y\in \spec R(\bar\bZ_p)$ 
equipping $U(\bar\bZ_p)$ with a $\Gal_F$-action,
such that the image of
$$H^1(\Gal_F, U(\bar\bZ_p)) \to H^1(\Gal_F, U(\bFp))$$
contains $[\bar c]$.
\end{egfthm}

\begin{egithm}
(Sufficiently small dimension)
If $\spec R_s\subset \spec (R \otimes \bFp)_{\red}$
is the closure of locus defined by
$$\dim_{\bFp}H^2(\Gal_F, U^{\ab}(\bFp))\ge s,$$
then $$\dim R_s + s \le \spec (R \otimes \bFp)_{\red}.$$
\end{egithm}

Here we define the dimension of an empty set to be $-\infty$
to avoid confusion.

\begin{thm}
\label{thm:EG}
(1) Property SSD implies Property EPL.

(2) For $\lsup LU_n$, $\GSO_n$
and $\GSp_n$, Property SSD holds when $p>2$.
\end{thm}

\subsection{Heisenberg-type extensions}~

For $G=\GL_n$,
we can choose $P$ such that $U=U^{\ab}$.
The good news is that for classical groups,
we can always choose $P$ such that
$U$ is the next best thing after abelian groups,
namely, unipotent algebraic groups of nilpotency class $2$.

\begin{thm}
\label{thm:classify}
(Lemma \ref{lem:unitary-1}, Lemma \ref{lem:symplectic-1}, 
Lemma \ref{lem:orth-1},
Theorem \ref{thm:main-unitary}, Theorem \ref{thm:main-symplectic}, Theorem \ref{thm:main-orthogonal})
Let $\lsup LG$ be any of $\lsup LU_n, \GSp_{2n}$ or $\GSO_{n}$.

(1) Each mod $p$ Langlands parameter
$\bar\rho:\Gal_F \to \lsup LG(\bFp)$
is either elliptic, 
or factors through a maximal proper parabolic
$\lsup LP$ 
such that $\bar\rho$
is a Heisenberg-type extension (see Definition \ref{def:H-type})
of some $\bar\rho_M:\Gal_F\to \lsup LM(\bFp)$
where $\lsup LM$ is the Levi factor of $\lsup LP$.

(2)
If $\bar\rho$ is a Heisenberg-type extension
of $\rho_M:\Gal_F\to \lsup LM(\bFp)$,
then Property EPL implies Property EL.
\end{thm}

A {\it Heisenberg-type extension}
is, roughly speaking,
an extension which has the least amount of ``non-linearity''.
More precisely, if $\bar\rho_P:\Gal_F\to \lsup LP(\bFp)$
is a Heisenberg-type extension of $\bar\rho_M$, then
$[U,U]$ is an abelian group, and
$$
\dim_{\bFp}H^2(\Gal_K, [U, U](\bFp))\le 1.
$$

The key technical result of this paper
is Theorem \ref{thm:Galois}.
Roughly speaking,
for Heisenberg-type extensions,
the non-linear part of the obstruction for lifting
is so mild that it can be killed through manipulating
cup products.
To make this idea work, we need
a resolution of Galois cohomology
supported on degrees $[0,2]$,
which is compatible with cup products
on the cochain level.
In this paper, the resolution used
is the Herr complexes.
Although Herr complexes are infinite-dimensional
resolutions,
we can truncate them to a finite system
while still retaining the structure of cup products.
The Heisenberg equations are defined through
cup products on the truncated Herr cochain groups.

This paper is organized as follows:
in Section 1-6,
we work out the group-insensitive machinery
of crafting lifts;
in Section 7-10,
we specialize to classical groups.

\newpage
\phantomsection
\addcontentsline{toc}{part}{Part I: General theory}

\noindent
{\bf \large Part I: General theory}

\section{Heisenberg equations}~

Let $r, s, t\in \bZ_+$ be integers.
Let $\Lambda$ be a DVR with uniformizer $\varpi$.
Let $d\in \Mat_{s\times t}(\Lambda)$
and $\Sigma_1,\dots,\Sigma_s\in \Mat_{r\times r}(\Lambda)$
be constant matrices.
For ease of notation, for $x\in \Mat_{r\times 1}(\Lambda)$, write
$$
x^t \Sigma x :=
\begin{bmatrix}
x^t \Sigma_1 x \\
\dots\\
x^t \Sigma_s x
\end{bmatrix}
\in \Mat_{s \times 1}(\Lambda).
$$
Here $x^t$ denotes the transpose of $x$.
We are interested in solving systems of equations in $(r+t)$ variables of the form
\[
x^t \Sigma x + d y = 0
\tag{$\dagger$}
\]
where $x\in \Lambda^{\oplus r}$ and
$y\in \Lambda^{\oplus t}$
are the $(r+t)$ variables.
We will call ($\dagger$)
the quadratic equation with coefficient matrix $(\Sigma,d)$.

\subsection{Lemma}
\label{lem:Heisenberg-1}
Let $\Lambda$ be a DVR with uniformizer $\varpi$.
Let $M$ be a finite flat $\Lambda$-module.
If $N\subset M$ is a submodule such that $M/N\cong \Lambda/\varpi^n$ ($n>0$),
then there exists a $\Lambda$-basis $\{x_1, x_2, \dots, x_s\}$
of $M$ such that
$N = \Span(\varpi^n x_1, x_2, \dots, x_s)$.

\begin{proof}
Let $\{e_1,\dots, e_s\}$ be a $\Lambda$-basis of $M$
and let $\{f_1,\dots,f_s\}$ be a $\Lambda$-basis of $N$.
There exists a matrix $X\in \GL_{s}(\Lambda[1/\varpi])$
such that $(f_1|\dots|f_s) = X (e_1|\dots|e_s)$.
By the theory of Smith normal form,
$X = S D T$ where $S,T\in \GL_{s}(\Lambda)$
and $D$ is a diagonal matrix.
We have $D = \Diag(\varpi^n,1,\dots,1)$.
Set $(x_1|\dots|x_s):=T(e_1|\dots|e_s)$,
and we are done.
\end{proof}

\subsection{Definition}
\label{def:Heisenberg-1}
A quadratic equation with
coefficient matrix $(\Sigma, d)$
is said to be {\it Heisenberg}
if 
\begin{itemize}
\item[(H1)] $\coker d\cong \Lambda$ or $\Lambda/\varpi^n$; and
\item[(H2)] there exists $f\in \Lambda^{\oplus r}$
such that $f^t \Sigma f \ne 0$ mod $(\varpi, \Img(d))$.
\end{itemize}

\subsection{Theorem}
\label{thm:Heisenberg-1}
Let $(\Sigma, d)$ be a Heisenberg equation over $\Lambda$.
If there exists a mod $\varpi$ solution $(\bar x, \bar y)\in (\Lambda/\varpi)^{\oplus r+t}$
to $(\Sigma, d)$ (that is, $\bar x^t \Sigma \bar x + d \bar y \in \varpi \Lambda^{\oplus s}$),
then there exists an extension of DVR $\Lambda\subset \Lambda'$
such that there exists a solution $(x, y)\in \Lambda^{\prime \oplus r+t}$
of $(\Sigma, d)$ lifting $(\bar x, \bar y)$.

\begin{proof}
Write $\{e_1,\dots,e_s\}$ for the standard basis for $\Lambda^{\oplus s}$.
By Lemma \ref{lem:Heisenberg-1},
we can assume $\Img(d)=\Span(\varpi^n e_1,e_2,\dots,e_s)$
or $\Span(e_2,\dots,e_s)$.
Write
$d=
\begin{bmatrix}
d_1 \\ \dots \\d_s
\end{bmatrix}
$.
By Definition \ref{def:Heisenberg-1},
there exists an element $f\in \Lambda^{\oplus r}$
such that $f^t\Sigma f \ne 0$ mod $(\varpi, \Img(d))$;
equivalently, $f^t\Sigma_1 f \ne 0$ mod $\varpi$.
Let $(x, y)$ be an arbitrary lift of $(\bar x, \bar y)$.
Let $\Lambda'$ be the ring of integers of the algebraic closure of
$\Lambda[1/\varpi]$.
Let $\lambda\in \Lambda'$.
Consider
\begin{eqnarray*}
(x + \lambda f)^t\Sigma_1(x + \lambda f)
+ d_1 y &=
(f^t\Sigma_1 f) \lambda^2 + (x^t\Sigma_1 f + f^t\Sigma_1 x) \lambda
+ (x^t\Sigma_1 x + d_1y);
\end{eqnarray*}
note that the $\varpi$-adic valuation of the quadratic term is $0$
while the $\varpi$-adic valuation of the constant team is positive.
By inspecting the Newton polygon, the quadratic equation above admits a solution
$\lambda \in \Lambda'$ of positive $\varpi$-adic valuation.
By replacing $x$ by $x+ \lambda f$,
we can assume 
$$
x^t \Sigma_1 x + d_1 y =0.
$$
Equivalently,
$x^t \Sigma x + d y \in \Span (e_2,\dots,e_s)\subset \Img(d)$.
By replacing $\Lambda$ by $\Lambda[\lambda]$,
we may assume $(x,y)\in \Lambda^{\oplus r + t}$.
In particular, there exists an element $z\in \Lambda^{\oplus t}$
such that 
$$
x^t \Sigma x + d y = d z,
$$
and it remains to show we can ensure $z=0$ mod $\varpi$.
We do know $d z = 0$ mod $\varpi$.
So $d z \in \Span(\varpi e_2,\dots, \varpi e_s)$.
Say $d z = \varpi u$, we have
$u\in \Span(e_2,\dots, e_s)\subset \Img(d)$.
Say $u=d v$.
So $d z = \varpi d v = d \varpi v$.
By replacing $z$ by $\varpi v$,
we have
$$
x^t \Sigma x + d y = d \varpi v.
$$
Finally, replacing $y$ by $(y-\varpi v)$,
we are done.
\end{proof}

We will call affine varieties defined by Heisenberg equations {\it Heisenberg varieties}.
Theorem \ref{thm:Heisenberg-1} says
all $\Lambda/\varpi$-points of a Heisenberg variety
admit a $\varpi$-adic thickening.

\section{Extensions of $(\varphi, \Gamma)$-modules and non-abelian $(\varphi, \Gamma)$-cohomology}~

Fix a pinned split reductive group $(\wh G, \wh B, \wh T, \{Y_\alpha\})$
over $\bZ$.
Fix a parabolic $P\subset \wh G$ containing $\wh B$
with Levi subgroupn $\wh M$
and unipotent radical $U$.
Let $K$ be a $p$-adic field.
Let $A$ be a $\bZ_p$-algebra.
The ring $\bA_{K, A}$ is defined as in 
\cite[Definition 4.2.8]{L23B}.
See \cite[Section 4.2]{L23B} for the definition of the procyclic group
$H_K$.
Fix a topological generator $\gamma$ of $H_K$.
Note that $\bA_{K, A}$ admits a Frobenius action $\varphi$
which commutes $\gamma$.

\subsection{Framed parabolic $(\varphi, \Gamma)$-modules}~
A {\it framed $(\varphi, \gamma)$-module} with $P$-structure
and $A$-coefficients is
a pair of matrices $[\phi], [\gamma]\in P(\bA_{K, A})$,
satisfying 
$[\phi]\varphi([\gamma])=[\gamma]\gamma([\phi])$.

A {\it framed $(\varphi, \Gamma)$-module} with $P$-structure
and $A$-coefficients
is a framed $(\varphi, \gamma)$-module $([\phi], [\gamma])$
with $P$-structure
and $A$-coefficients
such that there exists a closed algebraic group embedding
$P\hookrightarrow \GL_d = \GL(V)$ and
$1-[\gamma]$ induces a topologically nilpotent
$\gamma$-semilinear endomorphism of $V(\bA_{K, A})$.
To make is concrete, if $\{e_1,\dots,e_d\}$ is
the standard basis of $V$, then
$[\gamma]$ sends $\alpha e_i$ to $\gamma(\alpha)[\gamma]e_i$.

\subsection{Levi factor of $(\varphi, \gamma)$-modules}
Let $([\varphi], [\gamma])$ be a framed $(\varphi, \gamma)$-module
with $P$-structure and $A$-coefficients.
Write $([\varphi]_{\wh M}, [\gamma]_{\wh M})$
for its image under the projection $P\to \wh M$.
Note that $([\varphi]_{\wh M}, [\gamma]_{\wh M})$
is a framed $(\varphi, \gamma)$-module with $\wh M$-structure.

\subsection{Lemma}
\label{lem:ext-1}
Let $([\varphi], [\gamma])$ be a framed $(\varphi, \gamma)$-module
with $P$-structure and $A$-coefficients.
Then $([\varphi], [\gamma])$ is a framed $(\varphi, \Gamma)$-module
with $P$-structure  and $A$-coefficients
if and only if 
$([\varphi]_{\wh M}, [\gamma]_{\wh M})$ is a framed $(\varphi, \Gamma)$-module
with $\wh M$-structure  and $A$-coefficients.

\begin{proof}
Note that $P = U \rtimes \wh M$ and $\wh M$ is a subgroup of $P$.
We will regard both $P$ and $\wh M$ as a subgroup of $\GL_d\subset \Mat_{d\times d}$
by fixing an embedding $P\hookrightarrow \GL_d$.

Write $[u]:=[\gamma]-[\gamma]_{\wh M}$.
Note that the Jordan decomposition of $[\gamma]$ and $[\gamma]_{\wh M}$
has the same semisimple part;
write $[\gamma]=g_s g_u$
and $[\gamma]_{\wh M}=g_s g_u'$
for the Jordan decomposition where both $g_u$ and $g_u'$
lies in the unipotent radical of a Borel of $P$
(if we replace $\wh M$ by one of its conjugate in $P$, then $g_u$ and $g_u'$
lies in the unipotent radical of the same Borel of $P$).
By the Lie-Kolchin theorem, $(g_u-g_u')$ is nilpotent.
Since $g_s$ commutes with $g_u$ and $g_u'$, $[u]$ is nilpotent.
We have $(1-[\gamma]_{\wh M}) = (1-[\gamma] + [u])$.
Since $[u]$ is nilpotent, $(1-[\gamma]_{\wh M})$
is topologically nilpotent if and only if
$(1-[\gamma])$ is topologically nilpotent.
\end{proof}

\subsection{Extensions of $(\varphi, \Gamma)$-modules}~
In this paragraph, we classify
all framed $(\varphi, \Gamma)$-modules
with $P$-structure  and $A$-coefficients
whose Levi factor
is equal to a fixed $(\varphi, \Gamma)$-module with $\wh M$-structure
$([\phi]_{\wh M}, [\gamma]_{\wh M})$.

For ease of notation, write $f =[\phi]_{\wh M}$
and $g = [\gamma]_{\wh M}$.
We denote by
$$
H^1_{\Herr}(f, g)
$$
the equivalence classes of all extensions
of $(f, g)$ to a framed $(\varphi, \Gamma)$-module
with $P$-structure.

Let $u_f, u_g\in U(\bA_{K, A})$.
Set
$[\phi]=u_f f$ and $[\gamma]=u_g g$.
Note that $([\phi], [\gamma])$
is a $(\varphi, \Gamma)$-module
if and only if 
$$
u_f \Int_g(\gamma(u_f^{-1})) = 
u_g \Int_f(\varphi(u_g^{-1}))
$$
by Lemma \ref{lem:ext-1}.

Here $\Int_{?}(*) = ? * ?^{-1}$.

\subsection{Assumption}
\label{ass:ext}
We assume $U$ is a unipotent algebraic group
of nilpotency class $2$ and $p\ne 2$.
Assume there exists an embedding 
$\iota:U\hookrightarrow \GL_N$ such that
$(\iota(x)-1)^2=0$ for all $x\in U$.

In particular,
there is a well-defined truncated log map
\begin{eqnarray*}
\log:U \to&\Lie U \\
u \mapsto&  (u-1) - \frac{(u-1)^2}{2}
\end{eqnarray*}
whose inverse is the truncated exponential map
$$
\exp: \Lie U\to U.
$$
Here we embed $U$ into $\Mat_{d\times d}$ in order to define
addition and substraction
(the embedding is not important).
Write $x = \log(u_f)$ and $y = \log (u_g)$.

\subsection{Lemma}
$([\phi], [\gamma])$ is a $(\varphi, \Gamma)$-module
with $P$-structure
extending $([\phi]_{\wh M}, [\gamma]_{\wh M})$
if and only if
$$
(1 - \Int_g\circ \gamma)(x) - \frac{1}{2}[x, \Int_g \gamma(x)] =
(1 - \Int_f \circ\varphi)(y) - \frac{1}{2}[y, \Int_f \varphi(y)].
$$

\begin{proof}
It follows from the Baker-Campbell-Hausdorff formula.
\end{proof}

Recall that a nilpotent Lie algebra of nilpotency class $2$
is isomorphic to its associated graded Lie algebra
(with respect to either the lower or the upper central filtration).
We fix such an isomorphism $\Lie U \cong \gr^\bullet \Lie U = \gr^1\Lie U \oplus \gr^0\Lie U$
where $\gr^0\Lie U$ is the derived subalgebra of $\Lie U$.
Note that $\gr^0\Lie U$ is contained in the center of $\Lie U$.
In particular, if $x\in \Lie U$,
we can write $x = x_0+ x_1$ where $x_i\in \gr^i\Lie U$.

\subsection{Lemma}
\label{lem:Herr-2}
$([\phi], [\gamma])$ is a $(\varphi, \Gamma)$-module
extending $([\phi]_{\wh M}, [\gamma]_{\wh M})$
if and only if
$$
\begin{cases}
(1 - \Int_g\circ \gamma)(x_1) - (1 - \Int_f \circ\varphi)(y_1) = 0 \\
(1 - \Int_g\circ \gamma)(x_0) - (1 - \Int_f \circ\varphi)(y_0) = \frac{1}{2}[x_1, \Int_g \gamma(x_1)]
 - \frac{1}{2}[y_1, \Int_f \varphi(y_1)].
\end{cases}
$$

\begin{proof}
Arrange terms according to their degree
in the graded Lie algebra.
\end{proof}

\subsection{Herr complexes}
\label{par:Herr}
Let $V$ be a vector space over $\spec A$,
and let $(s, t)$ be a framed $(\varphi, \Gamma)$-module
with $\GL(V)$-structure  and $A$-coefficients.
Then the Herr complex associated to $(s, t)$ is by definition
the following
$$
C^\bullet_{\Herr}(s, t):=[V(\bA_{K, A}) \xrightarrow{(s\circ \varphi-1, t\circ \gamma-1)} V(\bA_{K, A}) \oplus V(\bA_{K, A})
\xrightarrow{(t\circ \gamma-1, 1-s\circ \varphi)^t} V(\bA_{K, A})]
$$
Write $Z^\bullet_{\Herr}(s, t)$,
$B^\bullet_{\Herr}(s, t)$
and $H^\bullet_{\Herr}(s, t)$
for the cocycle group, the coboundary group
and the cohomology group of $C^\bullet_{\Herr}(s, t)$.
The reader can easily check that our definiton is consistent with
that of \cite[Section 5.1]{EG23}.
We will denote by $d$ the differential operators in
$C^\bullet_{\Herr}(f, g)$.

Note that $\wh M$ acts on $\gr^1\Lie(U)$
and $\gr^0\Lie(U)$ by conjugation.
Write
\begin{align*}
\Int^0: \wh M \to \GL(\gr^0\Lie(U)),\\
\Int^1: \wh M \to \GL(\gr^1\Lie(U))
\end{align*}
for the conjugation actions.

\subsection{Cup products}
\label{par:cup}
Define a map
\begin{align*}
Q: C^1_{\Herr}(\Int^1(f, g)) & \to C^2_{\Herr}(\Int^0(f, g))\\
(x_1,y_1) &\mapsto \frac{1}{2}[x_1, \Int_g^1 \gamma(x_1)]
 - \frac{1}{2}[y_1, \Int_f^1 \varphi(y_1)]
\end{align*}
and a symmetric bilinear pairing
\begin{align*}
\cup: C^1_{\Herr}(\Int^1(f, g)) \times C^1_{\Herr}(\Int^1(f, g))
&\to C^2_{\Herr}(\Int^0(f, g)) \\
((x_1, y_1), (x_1',y_1')) & \mapsto
\frac{1}{2}(Q(x_1+x_1', y_1+y_1') - Q(x_1,y_1)-Q(x_1',y_1')).
\end{align*}

\subsection{Proposition}
\label{prop:Herr-1}
Define
\begin{align*}
Z_{\Herr}^1(f, g):=
\{
(x_0+x_1,y_0+y_1)\in &C^1_{\Herr}(\Int^0(f, g)\oplus \Int^1(f,g))
|\\
&
\begin{cases}
(x_0, y_0)\in C^1_{\Herr}(\Int^0(f, g))\\
(x_1, y_1)\in C^1_{\Herr}(\Int^1(f, g))\\
d(x_1,y_1) = 0\\
d(x_0, y_0) + (x_1,y_1)\cup (x_1,y_1)=0
\end{cases}
\}.
\end{align*}
There exists a surjective map
\begin{align*}
Z_{\Herr}^1(f,g) &\to H^1_{\Herr}(f, g) \\
(x,y)&\mapsto (\exp(x)f, \exp(y) g)
\end{align*}

\begin{proof}
It is a reformulation of Lemma \ref{lem:Herr-2}.
\end{proof}

\subsection{Lemma}
\label{lem:Herr-3}
The cup product induces a 
well-defined symmetric bilinear pairing
$$
\cup_H: H^1_{\Herr}(\Int^1(f, g))
\times H^1_{\Herr}(\Int^1(f, g))
\to H^2_{\Herr}(\Int^0(f, g)).
$$

\begin{proof}
The proof is formally similar to \cite[Lemma 2.3.9]{L25A}.
\end{proof}

\subsection{Non-split groups}
\label{rem:non-split}
We remark that all results in this section
holds for non-split groups.
More precisely,
Let $F\subset K$ be a $p$-adic field
and fix an action of $\Delta:=\Gal(K/F)$
on the pinned group $(\wh G, \wh B, \wh T, \{Y_\alpha\})$
and assume both $\wh M$ and $P$ are $\Delta$-stable.

Set $\lsup LG := \wh G \rtimes \Delta$,
$\lsup LP:= P\rtimes \Delta$,
and $\lsup LM \rtimes \Delta$.
Denote by $\GL^!(\Lie U)$
the (parabolic) subgroup of the general linear
    group $\GL(\Lie U)$ that preserves the lower central filtration
    of $\Lie U$.
Using the truncated log/exp map,
we have a group scheme homomorphism
$\lsup LM \to \GL^!(\Lie U)$,
which extends to a group scheme homomorphism
$$
\lsup LP = \lsup LM \rtimes U \to \GL^!(\Lie U) \rtimes U.
$$
Name the group $\GL^!(\Lie U) \rtimes U$
as $\wt P$,
and name the homomorphism
$\lsup LP \to \wt P$
as $\Xi$.

\subsection{Definition}
In the non-split setting,
a framed $(\varphi, \Gamma)$-module
with $\lsup LP$-structure
is a $(\varphi, \Gamma)$-module $(F, \phi_F, \gamma_F)$ with $\lsup LP$-structure,
and a framed $(\varphi, \Gamma)$-module
$([\phi], [\gamma])$ with $\wt P$-structure,
together with an identification
$\Xi_*(F, \phi_F, \gamma_F)\cong ([\phi], [\gamma])$.

\vspace{3mm}

The reason we make the definition above is because
$(\varphi, \Gamma)$-modules with $H$-structure
are not represented by a pair of matrices
if $H$ is a {\it disconnected} group.
So by choosing the map $\lsup LP\to \wt P$,
we are able to work with connected groups $\wt P$.

Since the whole purpose of this section is to
understand extensions of $(\varphi, \Gamma)$-modules
and
we fix the $\lsup LM$-semisimplification
of framed $(\varphi, \Gamma)$-module
with $\lsup LP$-structure,
the reader can easily see all results carry over to the non-split case
by using $\wt P$ in place of $P$.

\section{Cohomologically Heisenberg lifting problems}~

We keep notations from the previous section.
Let $\Lambda\subset \bar\bZ_p$ be a DVR.
Let $([\bar \phi], [\bar \gamma])$
be a framed $(\varphi, \Gamma)$-module
with $P$-structure and $\bFp$-coefficients.
Write $(\bar f, \bar g)$ for $([\bar \phi]_{\wh M}, [\bar \gamma]_{\wh M})$.
Fix a framed $(\varphi, \Gamma)$-module
$(f, g)$ with $\wh M$-structure and $\Lambda$-coefficients
lifting $(\bar f, \bar g)$.

By Proposition \ref{prop:Herr-1},
there exists an element
$(\bar x, \bar y)\in Z^1_{\Herr}(\bar f, \bar g)$
representing $([\bar \phi], [\bar \gamma])$.
We can write $\bar x=\bar x_0 + \bar x_1$
and $\bar y =\bar y_0+\bar y_1$
such that 
$(\bar x_i, \bar y_i)\in C^1_{\Herr}(\Int^i(\bar f, \bar g))$.

\subsection{Definition}
\label{def:coh}
A {\it cohomologically Heisenberg lifting problem}
is a tuple
$(f, g, \bar x, \bar y, H)$
consisting of
\begin{itemize}
\item a framed $(\varphi, \Gamma)$-module with $\wh M$-structure and $\Lambda$-coefficients $(f, g)$;
\item an element $(\bar x, \bar y)\in Z^1_{\Herr}(\bar f, \bar g)$, and
\item a $\Lambda$-submodule $H\subset H^1_{\Herr}(\Int^1(f, g))$,
\end{itemize}
such that
\begin{itemize}
\item[(HL1)] $(\bar x_1, \bar y_1)$ lies in the image of
$H$ in $H^1_{\Herr}(\Int^1(\bar f, \bar g))$,
\item[(HL2)] the pairing $\cup_H|_H: H \times H \to H^2_{\Herr}(\Int^0(f, g))$ is surjective,
\item[(HL3)] $H^2_{\Herr}(\Int^0(\bar f, \bar g))\cong \Lambda/\varpi$.
\end{itemize}
A {\it solution} to the lifting problem
$(f, g, \bar x, \bar y, H)$
is an element $(x, y)\in Z^1_{\Herr}(f, g)$
lifting $(\bar x, \bar y)$
such that
the image of $(x_1, y_1)$ in
$H^1_{\Herr}(\Int^1(f, g))$
is contained in $H$.

\vspace{3mm}

At first sight, a cohomologically Heisenberg lifting problem
defines an infinite system of quadratic polynomial equations.
In the following theorem, we show that we may truncate
the infinite system to a finite system and solve cohomologically
Heisenberg lifting problems.

\subsection{Theorem}
\label{thm:coh}
Each cohomologically Heisenberg lifting problem is solvable
after replacing $\Lambda$ by an extension of DVR $\Lambda'\subset \bar\bZ_p$.

\begin{proof}
Before we start, we remark that
$C^\bullet_{\Herr}(\Int^i(f, g))\otimes_{\Lambda}\Lambda/\varpi=C^\bullet_{\Herr}(\Int^i(\bar f, \bar g))$
while it is {\it not} generally true that
$H^\bullet_{\Herr}(\Int^i(f, g))\otimes_{\Lambda}\Lambda/\varpi=H^\bullet_{\Herr}(\Int^i(\bar f, \bar g))$.

Write $Z_H$ for the preimage of $H$ in $Z^1_{\Herr}(\Int^1(f, g))$.
Since $Z^1_{\Herr}(\Int^1(f, g))$ is $\Lambda$-torsion-free,
so is $Z_H$.
Let $X\subset Z_H$ be a finite $\Lambda$-submodule
which maps surjectively onto $H$.
Such an $X$ exists because $H^1_{\Herr}(\Int^1(f, g))$
is a finite $\Lambda$-module (\cite[Theorem 5.1.22]{EG23}).
Since $\Lambda$ is a DVR, $X$ is finite free over $\Lambda$.

Let $W\subset C^2_{\Herr}(\Int^0(f, g))$
be a finite free $\Lambda$-submodule
containing $X\cup X$.

By (HL2), $W\to H^2_{\Herr}(\Int^0(f, g))$
is surjective.
Set $B_W:=B^2_{\Herr}(\Int^0(f, g))\cap W$,
we have $H^2_{\Herr}(\Int^0(f, g)) = W/B_W$.
Again, $B_W$ is a finite free $\Lambda$-module.

Finally, let $Y\subset C^1_{\Herr}(\Int^0(f, g))$
be a finite free $\Lambda$-submodule
which maps surjectively onto $B_W$
and contains at least one lift of $(\bar x_0, \bar y_0)$.

Now consider the system of equations
\[
\fx \cup \fx + d \fy = 0 \in W
\tag{$\dagger$}
\]
where $\fx\in X$ and $\fy\in Y$
are the variables and $W$ is the value space.
We check that ($\dagger$) is a Heisenberg equation in the sense of
Definition \ref{def:Heisenberg-1}.
(H1) follows from (HL3) and the Nakayama lemma,
while (H2) follows from (HL2).
The equation ($\dagger$) admits a mod $\varpi$ solution
$\bar\fx,\bar\fy$
defined by $(\bar x, \bar y)\in Z^1_{\Herr}(\bar f, \bar g)$.
By Theorem \ref{thm:Heisenberg-1},
($\dagger$) admits a solution lifting
$\bar\fx,\bar\fy$
after extending the coefficient ring $\Lambda$.
The solution to the equation ($\dagger$)
is also a solution to the lifting problem.
\end{proof}

\section{Applications to Galois cohomology}~

Let $F/\Qp$ be a $p$-adic field,
and let $G$ is a tamely ramified quasi-split reductive group over $F$
which splits over $K$.
Write $\Delta:=\Gal(K/F)$.
There exists a $\Delta$-stable pinning $(G, B, T, \{X_\alpha\})$
of $G$,
and let $(\wh G, \wh B, \wh T, \{Y_\alpha\})$
be the dual pinned group.
Let $P\subset \wh G$ be a $\Delta$-stable parabolic of $\wh G$
with $\Delta$-stable Levi subgroup $\wh M$
and unipotent radical $U$.
Denote by $\lsup LP$ the semi-direct product $U\rtimes \lsup LM$
where $\lsup LM = \wh M \rtimes \Delta$.

In the terminology of \cite{L23A}, $\lsup LP$
is a big pseudo-parabolic of $\lsup LG = \wh G\rtimes \Delta$
and all big pseudo-parabolic of $\lsup LG$
are of the form $\lsup LP$
(see \cite[Section 3]{L23A}).

We enforce Assumption \ref{ass:ext} throughout this section.
Note that $\lsup LM$ acts on $\Lie U = \gr^0\Lie U \oplus \gr^1\Lie U$
by adjoint, and we denote the adjoint actions by
$\Int^i$ as in Paragraph \ref{par:Herr}.

\subsection{Definition}
\label{def:H-type}
Let $\bar\rho_P: \Gal_F\to \lsup LP(\bFp)$
be a Langlands parameter with Levi factor
$\bar\rho_M:\Gal_F\to \lsup LM(\bFp)$.
We say $\bar\rho_P$ is a {Heisenberg-type} extension of $\bar\rho_M$
if $$\dim_{\bFp}H^2(\Gal_K, \gr^0\Lie U(\bFp)) \le 1.$$
Here the $\Gal_F$-action on $\gr^0\Lie U(\bFp)$
is obtained from composing $\bar\rho_M$
and $\Int^0:\lsup LM \to \GL(\gr^0\Lie U)$.

\subsection{Cup products on Galois cohomology}
Let $A$ be either $\bFp$ or $\bar\bZ_p$.
If $\rho_M:\Gal_F\to \lsup LM(A)$ is an $L$-parameter,
we can equip $\Lie U(A)$ with $\Gal_F$-action
via $\rho_M$.
Note that there exists a symmetric bilinear pairing
$$
H^1(\Gal_F, \gr^1\Lie U(A))
\times
H^1(\Gal_F, \gr^1\Lie U(A))
\to H^2(\Gal_F, \gr^0\Lie U(A)),
$$
which is defined in \cite[Section 3.2]{L25A}.
Alternatively, we can transport the symmetric cup product
on $(\varphi, \Gamma)$-cohomology defined in Definition \ref{par:cup}
and Lemma \ref{lem:Herr-3},
and later generalized in \ref{rem:non-split}
to Galois cohomology.

\subsection{Partial extensions and partial lifts}
A {\it partial extension of $\rho_M$}
is a continuous group homomorphism
$\rho':\Gal_F\to \frac{\lsup LP}{[U,U]}(\bar\bZ_p)$
extending $\rho_M:\Gal_F \to \lsup LM(\bar\bZ_p) = \frac{\lsup LP}{U}(\bar\bZ_p)$.
Here $[U,U]$ is the derived subgroup of $U$.
The set of equivalence classes of partial extensions of
$\rho_M$ are in natural bijection with
$H^1(\Gal_F, \frac{U}{[U,U]}(\bar\bZ_p))=H^1(\Gal_F, \gr^1\Lie U(\bar\bZ_p))$.

Let $\bar\rho_P:\Gal_F\to \lsup LP(\bFp)$ be an $L$-parameter.
A {\it partial lift} of $\bar\rho_P$
is a group homomorphism $\rho':\Gal_F\to \frac{\lsup LP}{[U,U]}(\bar\bZ_p)$
which lifts $\bar\rho_P$ mod $[U,U]$.

\subsection{Lemma}
\label{lem:partial}
A partial extension $\rho':\Gal_F\to \frac{\lsup LP}{[U,U]}(A)$
of $\rho_{M, A}$ extends to a full extension
$\rho:\Gal_F\to\lsup LP(A)$
if and only if $c \cup c = 0$
    where $c\in H^1(\Gal_F, \gr^1\Lie U(A))$
    is the cohomology class corresponding to $\rho'$.

\begin{proof}
It follows immediately from Proposition \ref{prop:Herr-1}.
\end{proof}

\subsection{Theorem}
\label{thm:Galois}
Assume $p\ne 2$.
Let $\bar\rho_P:\Gal_F\to \lsup LP(\bFp)$
be an extension of $\bar\rho_M$.
Assume $\bar\rho_M$ admits a lift $\rho_M:\Gal_F\to \lsup LM(\bar\bZ_p)$
such that
\begin{itemize}
\item[(i)] $\bar\rho_P$
is a Heisenberg-type extension of $\bar\rho_M$,
\item[(ii)] 
$\bar\rho_P|_{\Gal_K}$ admits a partial lift
which is a partial extension of
$\rho_M|_{\Gal_K}$, and
\item[(iii)]
the pairing
$$
H^1(\Gal_F, \gr^1\Lie U(\bar\bZ_p))\underset{\bar\bZ_p}{\otimes}\bFp
\times
H^1(\Gal_F, \gr^1\Lie U(\bar\bZ_p))\underset{\bar\bZ_p}{\otimes}\bFp
\to H^2(\Gal_F, \gr^0\Lie U(\bFp))$$
is non-trivial
unless $H^2(\Gal_F, \gr^0\Lie U(\bFp))=0$,
\end{itemize}
then $\bar\rho_P$
admits a lift $\rho_P:\Gal_F\to \lsup LP(\bar\bZ_p)$
with Levi factor $\rho_M$.

\begin{proof}
Let $A$ be either $\bFp$ or $\bar\bZ_p$,
and let $\rho_{M, A}$ be either $\bar\rho_M$ or $\rho_M$,
resp..
The set of equivalence classes
of $L$-parameters $\Gal_F\to \lsup LP/[U,U](A)$
extending $\rho_{M, A}$
is in natural bijection with
the $A$-module
$H^1(\Gal_F, \gr^1\Lie U(A))$.
Since $K/F$ is assumed to have prime-to-$p$ degree,
we have
$$
H^1(\Gal_F, \gr^1\Lie U(A)) = 
H^1(\Gal_K, \gr^1\Lie U(A))^\Delta
$$
by \cite[Theorem 3.15]{Ko02}.

The $L$-parameter $\bar\rho_P$ defines
an element $\bar c \in H^1(\Gal_K, \gr^1\Lie U(\bFp))^\Delta$.
By item (ii), there exists a lift
$c'\in H^1(\Gal_K, \gr^1\Lie U(\bar\bZ_p))$
lifting $\bar c$.
Define $c:=\frac{1}{[K:F]}\sum_{\gamma\in \Delta}\gamma c\in H^1(\Gal_F, \gr^1\Lie U(\bar\bZ_p))$.
It is clear that $c$
lifts $\bar c$.

There are two possibilities: either
$$\dim_{\bFp}H^2(\Gal_F, \gr^0\Lie U(\bFp)) =0,$$
or
$$\dim_{\bFp}H^2(\Gal_F, \gr^0\Lie U(\bFp)) =1.$$
In the former case, there is no obstruction to extension and lifting
and the Theorem follows from \cite[Proposition 5.3.1]{L25A}.
Now we consider the latter case.

By Fontaine's theory of $(\varphi, \Gamma)$-modules,
$\rho_M$ corresponds to a framed $(\varphi, \Gamma)$-module
$(f, g)$ with $\lsup LM$-structure (or rather $\Gal^!(\Lie U)$-structure by Paragraph \ref{rem:non-split}), and
$\bar\rho_P$ corresponds to an element
$(\bar x, \bar y)\in Z^1_{\Herr}(\bar f, \bar g)$.
Here $(\bar f, \bar g)$ is the reduction of $(f, g)$.
Since Galois cohomology is naturally isomorphic to the cohomology
of Herr complexes (\cite[Theorem 5.1.29]{EG23}),
we can identify $H^1_{\Herr}(\Int^i(f, g))$
with $H^1(\Gal_F, \gr^i\Lie U(\bar\bZ_p))$.

Consider the tuple $(f, g, \bar x, \bar y, H^1(\Gal_F, \gr^1\Lie U(\bar\bZ_p)))$.
We want to check that this tuple is a cohomologically Heisenberg lifting problem in the sense of Definition \ref{def:coh}.
(HL3) follows from assumption (i),
(HL2) follows from assumption (iii) and (HL3),
and (HL1) follows from assumption (ii) and the discussion in the second paragraph of this proof.
We finish the proof by invoking Theorem \ref{thm:coh}.
\end{proof}

\section{Interaction of cup products with  $\bZ/2$-action}~
\label{sec:cup}

Let $K$ be a $p$-adic field.
Let $a=c$ and $b$ be positive integers.
Fix a Galois representation
$$
\bar \tau
=
\begin{bmatrix}
\bar \tau_a & & \\
& \bar \tau_b & \\
& & \bar \tau_c
\end{bmatrix}
:\Gal_K \to
\begin{bmatrix}
\GL_a & & \\
& \GL_b & \\
& & \GL_c
\end{bmatrix}(\bFp),
$$
as well as 
a lift
$$
\tau=
\begin{bmatrix}
\tau_a & & \\
& \tau_b & \\
& & \tau_c
\end{bmatrix}
:\Gal_K \to
\begin{bmatrix}
\GL_a & & \\
& \GL_b & \\
& & \GL_c
\end{bmatrix}(\bar\bZ_p)
$$
of $\bar \tau$.
Write $\gr^0\Lie U := \Mat_{a\times c}$,
and $\gr^1\Lie U :=\Mat_{a\times b} \oplus \Mat_{b\times c}$.
Recall that we have defined a (symmetrized) cup product
$$
H^1(\Gal_K, \gr^1\Lie U(A))
\times
H^1(\Gal_K, \gr^1\Lie U(A))
\to 
H^2(\Gal_K, \gr^0\Lie U(A))
$$
for $A=\bFp, \bar\bZ_p$.
If $c\in H^i(\Gal_K, \gr^1\Lie U(A))$,
write $c = (c_1, c_2)$
where $c_1\in H^i(\Gal_K, \Mat_{a\times b}(A))$,
and $c_2\in H^i(\Gal_K, \Mat_{b\times c}(A))$,

\subsection{Lemma}
\label{lem:unitary-cup-1}
For the cup product
$$H^1(\Gal_K, \gr^1\Lie U(A)) \times H^1(\Gal_K, \gr^1\Lie U(A))
\to H^2(\Gal_K, \gr^0\Lie U(A)),$$
we have
\begin{align*}
(c_1, 0) \cup (c_1, 0) &= 0 \\
(0, c_2) \cup (0, c_2) &= 0\\
(c_1, 0) \cup (0, c_2) &= \frac{1}{2}(c_1, c_2) \cup (c_1, c_2)\\
(c_1, 0) \cup (c_1', 0) &= 0 \\
(0, c_2) \cup (0, c_2') &= 0,
\end{align*}
for any $c_1, c_2, c_1', c_2'$.
\begin{proof}
The first two identities follows from Lemma \ref{lem:partial}.
The last three identity follows from the first two identities.
\end{proof}

Write $\Delta$ for the finite group $\{1, \j\}$ with two elements.
While $\Delta$ denotes the Galois group $\Gal(K/F)$ in the previous sections,
$\Delta$ is merely an abstract group in this section.

\subsection{Definition}
\label{def:classical}
An action of $\Delta$ on each of
$H^i(\Gal_K, \gr^j\Lie U(A))$
is said to be {\it classical}
if
\begin{itemize}
\item
$\j H^i(\Gal_K, \Mat_{a\times b}(A))
\subset H^1(\Gal_K, \Mat_{b\times c}(A))$, $i=0,1,2$;
\item $\j H^i(\Gal_K, \Mat_{b\times c}(A))
\subset H^1(\Gal_K, \Mat_{a\times b}(A))$, $i=0,1,2$;
\item the $\Delta$-action is compatible with the cup products.
\end{itemize}
We also call such $\Delta$-actions
{\it classical $\Delta$-structure}.

\subsection{Proposition}
\label{prop:unitary-cup-1}
Fix a classical $\Delta$-structure.
If
$$H^2(\Gal_K, \gr^0\Lie U(\bFp))=H^2(\Gal_K, \gr^0\Lie U(\bFp))^\Delta=\bFp,$$
then the cup product
$$
H^1(\Gal_K, \gr^1\Lie U(\bar\bZ_p))^\Delta\underset{\bar\bZ_p}{\otimes}\bFp
\times
H^1(\Gal_K, \gr^1\Lie U(\bar\bZ_p))^\Delta\underset{\bar\bZ_p}{\otimes}\bFp
\to H^2(\Gal_F, \gr^0\Lie U(\bFp))^\Delta
$$
is non-trivial if and only if
$$
H^1(\Gal_K, \gr^1\Lie U(\bar\bZ_p))\underset{\bar\bZ_p}{\otimes}\bFp
\times
H^1(\Gal_K, \gr^1\Lie U(\bar\bZ_p))\underset{\bar\bZ_p}{\otimes}\bFp
\to H^2(\Gal_K, \gr^0\Lie U(\bFp))
$$
is non-trivial.

\begin{proof}
Since $\cup$ is a symmetric pairing,
non-triviality of $\cup$
on $H^1(\Gal_K, \gr^1\Lie U(\bar\bZ_p))$
implies $(c_1, c_2)\cup(c_1, c_2)\ne 0$
for some $(c_1, c_2)\in H^1(\Gal_K, \gr^1\Lie U(\bar\bZ_p))\underset{\bar\bZ_p}{\otimes}\bFp$.
By Lemma \ref{lem:unitary-cup-1},
$(c_1, 0)\cup (0, c_2)\ne 0$.
Since $$H^2(\Gal_K, \gr^0\Lie U(\bFp))^\Delta=H^2(\Gal_K, \gr^0\Lie U(\bFp)),$$
we conclude $\Delta$-acts trivially on $H^2(\Gal_K, \gr^0\Lie U(\bFp))$.

We argue by contradition and assume
$\cup$ is trivial on $H^1(\Gal_K, \gr^1\Lie U(\bar\bZ_p))^\Delta\underset{\bar\bZ_p}{\otimes}\bFp$.
We claim $x\cup y = 0$
for each $x\in H^1(\Gal_K, \gr^1\Lie U(\bar\bZ_p))^\Delta\underset{\bar\bZ_p}{\otimes}\bFp$ and
$y\in H^1(\Gal_K, \gr^1\Lie U(\bar\bZ_p))\underset{\bar\bZ_p}{\otimes}\bFp$.
Indeed,
$$2(x\cup y) = x\cup y + \j(x\cup y) = x\cup y + (\j x)\cup (\j y)
= x\cup y + x \cup \j y = x\cup (y+\j y)=0$$
because both $x$ and $y +\j y$ lies in
$H^1(\Gal_K, \gr^1\Lie U(\bar\bZ_p))^\Delta\underset{\bar\bZ_p}{\otimes}\bFp$.

Since $(c_1, 0) + \j (c_1, 0)\in H^1(\Gal_K, \gr^1\Lie U(\bar\bZ_p))^\Delta\underset{\bar\bZ_p}{\otimes}\bFp$,
we have
$((c_1, 0) + \j (c_1, 0)) \cup (0, c_2) = 0$.
However, since $(c_1, 0)\cup (0, c_2)\ne 0$,
we must have $\j (c_1, 0) \cup (0, c_2)\ne 0$.
By the classicality of the $\Delta$-structure,
$\j (c_1, 0) = (0, c_2')$ for some $c_2'$ and thus by Lemma \ref{lem:unitary-cup-1},
$\j(c_1, 0)\cup (0, c_2)=0$ and we get a contradiction.
\end{proof}

Next, we establish a general non-triviality of cup products
result,
and before that we need a non-degeneracy result.

\subsection{Lemma}
\label{lem:unitary-cup-3}
Assume $\bar\tau_a$ and $\bar\tau_c$ are irreducible.
Then the cup product
$$
H^1(\Gal_K, \Mat_{a\times b}(\bFp))
\times
H^1(\Gal_K, \Mat_{b\times c}(\bFp))
\to H^2(\Gal_K, \Mat_{a\times c}(\bFp))
$$
is non-degenerate.

\begin{proof}
Fix a non-zero element $x \in 
H^1(\Gal_K, \Mat_{b\times c}(\bFp))$.
Such an extension class $x$ corresponds to a non-split extension
$\bar\tau_d = 
\begin{bmatrix}
\bar\tau_b & * \\
& \bar\tau_c
\end{bmatrix}$.
In particular, the map
$$
H^0(\Gal_K, \bar\tau_a^\vee \otimes \bar\tau_d(1))
\to
H^0(\Gal_K, \bar\tau_a^\vee \otimes \bar\tau_c(1))
$$
is the zero map (if otherwise the socle of $\bar\tau_d$ is strictly larger than the socle of $\bar\tau_b$ and $\bar\tau_d$ must be a split extension).
By local Tate duality,
the map
$$
H^2(\Gal_K, \bar\tau_c^\vee \otimes \bar\tau_a)
\to
H^2(\Gal_K, \bar\tau_d^\vee \otimes \bar\tau_a)
$$
is also the zero map.
The short exact sequence
$$
0\to \bar\tau_b \to \bar\tau_d \to \bar\tau_c \to 0
$$
induces
the long exact sequence
$$
H^1(\Gal_K, \bar\tau_c^\vee \otimes \bar\tau_a)
\to
H^1(\Gal_K, \bar\tau_d^\vee \otimes \bar\tau_a)
\to
H^1(\Gal_K, \bar\tau_b^\vee \otimes \bar\tau_a)
\to
H^2(\Gal_K, \bar\tau_c^\vee \otimes \bar\tau_a)
\xrightarrow{0}
H^2(\Gal_K, \bar\tau_d^\vee \otimes \bar\tau_a).
$$
Since
$H^2(\Gal_K, \bar\tau_c^\vee \otimes \bar\tau_a) \ne 0$,
there exists an element
$y\in H^1(\Gal_K, \bar\tau_b^\vee \otimes \bar\tau_a)$
which maps to a non-zero element of $H^2(\Gal_K, \bar\tau_c^\vee \otimes \bar\tau_a)$,
and $y$ does not admit an extension
to $H^1(\Gal_K, \bar\tau_d^\vee \otimes \bar\tau_a)$.
By Lemma \ref{lem:partial},
$(x, y)\cup (x, y) \ne 0$,
and thus by Lemma \ref{lem:unitary-cup-1},
$x\cup y = \frac{1}{2}((x, y)\cup (x, y) )\ne 0$.
\end{proof}

\subsection{Lemma}
\label{lem:unitary-cup-4}
Let $X, Y$ be vector spaces over a field $\kappa$.
Let 
$$
\cup: X\times Y \to \kappa
$$
be a non-degenerate
bilinear pairing.
Let $H_X\subset X$ and $H_Y\subset Y$
be subspaces such that
$x\cup y=0$ for all $x\in H_X$ and $y\in H_Y$.
Then either $\dim X \ge 2\dim H_X$
or $\dim Y \ge 2 \dim H_Y$.

\begin{proof}
It suffices to show
$\dim X + \dim Y \ge 2(\dim H_X+\dim H_Y)$.
We define a (symmetric) inner product structure on
$X\oplus Y$ by setting
$(x, y)\cdot (x, y) = x \cup y + y\cup x$.
The lemma now follows from the Gram-Schmidt process.
\end{proof}

\subsection{Lemma}
\label{lem:unitary-cup-5}
Assume both $\bar\tau_a$ and $\bar\tau_c$ are irreducible.
The cup product
$$
H^1(\Gal_K, \Mat_{a\times b}(\bar\bZ_p))\underset{\bar\bZ_p}{\otimes}\bFp
\times
H^1(\Gal_K, \Mat_{b\times c}(\bar\bZ_p))\underset{\bar\bZ_p}{\otimes}\bFp
\to H^2(\Gal_K, \Mat_{a\times c}(\bFp))
$$
is non-trivial unless
all of the following holds
\begin{itemize}
\item 
$a=1$,
\item $K=\Qp$,
\item either $\bar\tau_b = \bar\tau_a(-1)^{\oplus b}$
and $\bar\tau_c = \bar\tau_a(-1)$; or
$\bar\tau_b = \bar\tau_a^{\oplus b}$
and $\bar\tau_c = \bar\tau_a(-1)$.
\end{itemize}

\begin{proof}
By Lemma \ref{lem:unitary-cup-4}
and Lemma \ref{lem:unitary-cup-3},
the lemma holds if
\[
\dim_{\bFp}
H^1(\Gal_K, \Mat_{a\times b}(\bar\bZ_p))\underset{\bar\bZ_p}{\otimes}\bFp
> \frac{1}{2} \dim_{\bFp}
H^1(\Gal_K, \Mat_{a\times b}(\bFp))
\tag{$*$}
\]
and
\[\dim_{\bFp}
H^1(\Gal_K, \Mat_{b\times c}(\bar\bZ_p))\underset{\bar\bZ_p}{\otimes}\bFp
> \frac{1}{2} \dim_{\bFp}
H^1(\Gal_K, \Mat_{b\times c}(\bFp)).
\tag{$**$}
\]
We prove by contradiction
and assume either ($*$) or ($**$) fails.
Since ($*$) and ($**$) are completely similar,
we assume the contraposition of ($*$)
that 
\begin{equation}
\label{eq:cup-1}
H^1(\Gal_K, \Mat_{a\times b}(\bar\bZ_p))\underset{\bar\bZ_p}{\otimes}\bFp
\le \frac{1}{2} \dim_{\bFp}
H^1(\Gal_K, \Mat_{a\times b}(\bFp)).
\end{equation}
By the universal coefficient theorem, we have
the short exact sequence
$$
0 \to H^1(\Gal_K, \Mat_{a\times b}(\bar\bZ_p))\underset{\bar\bZ_p}{\otimes}\bFp
\to H^1(\Gal_K, \Mat_{a\times b}(\bFp))
\to \Tor^{\bar\bZ_p}_1(H^2(\Gal_K, \Mat_{a\times b}(\bar\bZ_p)), \bFp)\to 0.
$$
Therefore the assumption ($\ref{eq:cup-1}$) is equivalent to
\begin{equation}
\label{eq:cup-2}
\dim_{\bFp}
H^1(\Gal_K, \Mat_{a\times b}(\bar\bZ_p))\underset{\bar\bZ_p}{\otimes}\bFp
\le
\dim_{\bFp}
\Tor^{\bar\bZ_p}_1(H^2(\Gal_K, \Mat_{a\times b}(\bar\bZ_p)), \bFp).
\end{equation}
Note that
$$
\dim_{\bFp}
\Tor^{\bar\bZ_p}_1(H^2(\Gal_K, \Mat_{a\times b}(\bar\bZ_p)), \bFp)
\le\dim_{\bFp}H^2(\Gal_K, \Mat_{a\times b}(\bFp)),
$$
since $H^2$ commutes with base change;
also see \cite[Example 3.1.7]{We94} for the computation of $\Tor$.
Also note that
$$
\dim_{\bFp}
H^1(\Gal_K, \Mat_{a\times b}(\bar\bZ_p))\underset{\bar\bZ_p}{\otimes}\bFp
\ge \rank_{\bar\bZ_p}
H^1(\Gal_K, \Mat_{a\times b}(\bar\bZ_p))_{\text{torsion-free}},
$$
and
$$
\rank_{\bar\bZ_p}
H^1(\Gal_K, \Mat_{a\times b}(\bar\bZ_p))_{\text{torsion-free}}
= \dim_{\bQp}H^1(\Gal_K, \Mat_{a\times b}(\bQp)).
$$
By local Euler characteristic, we have
\begin{align*}
\dim_{\bQp}H^1(\Gal_K, \Mat_{a\times b}(\bQp))
=& \dim_{\bQp}H^0(\Gal_K, \Mat_{a\times b}(\bQp))\\
&+ \dim_{\bQp}H^2(\Gal_K, \Mat_{a\times b}(\bQp))
+ [K:\Qp] a b\\
\ge & [K:\Qp] ab.
\end{align*}
On the other hand, by local Tate duality
$$
\dim_{\bFp}H^2(\Gal_K, \Mat_{a\times b}(\bFp))
=\dim_{\bFp}H^0(\Gal_K, \tau_a^\vee \otimes\tau_b(1))
\le a b.
$$
Combine all above, (\ref{eq:cup-2}) becomes
$$
[K:\Qp]a b 
\le a b.
$$
So, all inequalities above must be equalities and we are forced to have
\begin{itemize}
\item[(i)] 
$K=\Qp$,
\item[(ii)]
$H^1(\Gal_K, \Mat_{a\times b}(\bar\bZ_p))$ is torsion-free,
and 
\item[(iii)] $\dim_{\bFp}H^2(\Gal_K, \Mat_{a\times b}(\bFp))=ab$.
\end{itemize}
Item (iii) further forces $a=1$ because $\bar\tau_a$ is assumed to be irreducible.
\end{proof}

\subsection{Corollary}
\label{cor:unitary-cup-1}
Fix a classical $\Delta$-structure.
Assume both $\bar\tau_a$ and $\bar\tau_c$ are irreducible.
The cup product
$$
H^1(\Gal_K, \gr^1\Lie U(\bar\bZ_p))^\Delta\underset{\bar\bZ_p}{\otimes}\bFp
\times
H^1(\Gal_K, \gr^1\Lie U(\bar\bZ_p))^\Delta\underset{\bar\bZ_p}{\otimes}\bFp
\to H^2(\Gal_K, \Mat_{a\times c}(\bFp))^\Delta
$$
is non-trivial unless
all of the following holds
\begin{itemize}
\item 
$a=1$,
\item $K=\Qp$,
\item either $\bar\tau_b = \bar\tau_a(-1)^{\oplus b}$
and $\bar\tau_c = \bar\tau_a(-1)$; or
$\bar\tau_b = \bar\tau_a^{\oplus b}$
and $\bar\tau_c = \bar\tau_a(-1)$.
\end{itemize}

\begin{proof}
Combine Lemma \ref{lem:unitary-cup-5}
and Proposition \ref{prop:unitary-cup-1}.
\end{proof}

\newpage
\phantomsection
\addcontentsline{toc}{part}{Part II: Classical groups}

\noindent
{\bf \large Part II: Classical groups}

We use the convenient convention that
$\GSp_0 = \GO_0=\GL_1$, and $\Sp_0=\SO_0=U_0=1$.

Denote by $w_n=\begin{bmatrix}& & 1 \\& \cdots & \\1&&\end{bmatrix}$
the antidiagonal matrix whose non-zero entries
are $1$.

\section{A: unitary groups}~
\label{sec:ugrp}

Assume $G=U_n$ is a quasi-split tamely ramified unitary group over $F$
which splits over the quadratic extension $K/F$.
The $L$-group $\lsup LU_n = \GL_n \rtimes \{1, \iota\}$
where $\iota$ acts on $\GL_n$ via $w_n(-)^{-t}w_n$.

The Dynkin diagram of $G$ is a chain of $(n-1)$-vertices
(\dynkin{A}{}),
and $\Delta=\Gal(K/F)$ acts on $\Dyn(G)$ by reflection.
The maximal proper $\Delta$-stable subsets of $\Dyn(G)$
are given by removing either two symmetric vertices, or the middle vertex.
Therefore, the Levi subgroups of maximal proper $F$-parabolics
of $G$ are of the form
$$
M_k:=\Res_{K/F}\GL_k \times U_{n-2k}.
$$
If $\lsup LP$ is a maximal proper parabolic of $\lsup LG$,
then the Levi of $\lsup LP$ is
of the form $\lsup LM_k$;
we will write $\lsup LP_k$ for $\lsup LP$ to emphasize its type.

\subsection{Proposition}
\label{prop:unitary}
Let $\bar\rho: \Gal_F\to \lsup LG(\bFp)$ be an $L$-parameter.
Then either $\bar\rho$ is elliptic,
or $\bar\rho$ factors through
$\lsup LP_k(\bFp)$ for some $k$
such that the composite
$\bar r:\Gal_F \xrightarrow{\bar\rho^\Ss} \lsup LM_k(\bFp)
\to \lsup L\Res_{K/F}\GL_k(\bFp)$
is elliptic.

\begin{proof}
By \cite[Theorem B]{L23A},
$\bar\rho$ is either elliptic, or factors through
some $\lsup LP_k(\bFp)$.
By the non-abelian Shapiro's lemma,
$L$-parameters $\Gal_F\to \lsup L\Res_{K/F}\GL_k(\bFp)$
are in natural bijection to $L$-parameters
$\Gal_K\to \GL_k(\bFp)$;
and this bijection clearly preserves ellipticity.
Suppose $\bar r$ is not elliptic, then
$\bar r$, when regarded as a Galois representation
$\Gal_K\to \GL_k(\bFp)$,
contains a proper irreducible subrepresentation
$\bar r_0:\Gal_K\to \GL_s(\bFp)$.
It is easy to see that $\bar\rho$
also factors through $\lsup LP_s(\bFp)$. So we are done.
\end{proof}

We take a closer look at $\lsup LP_k$:
$$
\lsup LP_k=
\begin{bmatrix}
\GL_k & \Mat_{k\times (n-2k)} & \Mat_{k \times k} \\
& \GL_{n-2k} & \Mat_{(n-2k)\times k}\\
& & \GL_{k}
\end{bmatrix}
\rtimes \Delta
$$

\subsection{Lemma}
\label{lem:unitary-1}
If $\bar\rho:\Gal_F \to \lsup LG(\bFp)$
is not elliptic, then there exists a parabolic
$\lsup LP_k$ through which $\bar\rho$ factors
and $\bar\rho$ is a Heisenberg-type extension of
some $\bar\rho_M:\Gal_F \to \lsup LM_k(\bFp)$.

\begin{proof}
By Proposition \ref{prop:unitary},
there exists a parabolic $\lsup LP_k$ such that
$\bar r:\Gal_F \xrightarrow{\bar\rho^\Ss} \lsup LM_k(\bFp)
\to \lsup L\Res_{K/F}\GL_k(\bFp)$
is elliptic.
Write
$$
\bar r|_{\Gal_K} = 
\begin{bmatrix}
\bar r_1 & & \\
& 1_{n-2k} & \\
& & \bar r_2
\end{bmatrix},
$$
where $\bar r_1, \bar r_2:\Gal_K\to \GL_k(\bFp)$.
By the non-abelian Shapiro's lemma (see \cite[Subsection 9.4]{GHS} for details),
$\bar r$ can be fully reconstructed from $\bar r_1$,
and $\bar r_2$ is completely determined by $\bar r_1$;
in particular, both $\bar r_1$ and $\bar r_2$
are irreducible Galois representations.
We
have $$
H^2(\Gal_K, \gr^0\Lie U(\bFp))
=H^2(\Gal_K, \Hom(\bar r_2, \bar r_1)).
$$
Since both $\bar r_1$ and $\bar r_2$
are irreducible, by local Tate duality,
we have $\dim H^2(\Gal_K, \Hom(\bar r_2, \bar r_1))
= \dim H^0(\Gal_K, \Hom(\bar r_1, \bar r_2(1)))\le 1$.
\end{proof}

Next, we study cup products.
Fix the parabolic type $\lsup LP_k$.
We have
$$\gr^1\Lie U = \Mat_{k\times(n-2k)} \oplus \Mat_{(n-2k)\times k}$$
and
$$\gr^0\Lie U = \Mat_{k\times k}.$$
We will use all notations introduced in Section \ref{sec:cup}.

By \cite[Theorem 3.15]{Ko02},
we have
$$
H^i(\Gal_F, \gr^j\Lie U(A))
 = 
H^i(\Gal_K, \gr^j\Lie U(A))^{\Gal(K/F)}
=
H^i(\Gal_K, \gr^j\Lie U(A))^{\Delta},
$$
for all $i$ and $j$.

\subsection{Lemma}
\label{lem:unitary-cup-2}
The $\Delta=\Gal(K/F)$-action on $H^i(\Gal_K, \gr^1\Lie U(A))$
satisfies
\begin{align*}
\j(c_1, 0) = (0, *)\\
\j(0, c_2) = (*, 0),
\end{align*}
for any $c_1, c_2$.

\begin{proof}
The proof is completely similar for $i=0,1,2$,
and we only showcase the $i=1$ case.
Write
$$
w=\begin{bmatrix}
0 & 0 & J_1 \\
0 & J_2 & 0 \\
J_3 & 0 & 0
\end{bmatrix}
$$
for (a representative of) the longest Weyl group element.
Let
$$
\rho=
\begin{bmatrix}
A & B & * \\
& D & E \\
& & F
\end{bmatrix}:
\Gal_K\to P(A)
$$
be a group homomorphism.
Note that each of $A, B, D, E, F$
is a matrix-valued function on $\Gal_K$.
Write $A'$ for $\gamma\mapsto A(\j^{-1} \gamma \j)$
and similarly define $B', D', E', F'$.
We have
$$
\j \rho (\j^{-1} - \j) \j^{-1}
= w \rho(\j^{-1} - \j)^{-t} w^{-1}
=
\begin{bmatrix}
J_1 F^{\prime-t} J_1^{-1} & -J_1 F^{\prime-t}E^{\prime t} D^{\prime -t}
J_2^{-1} & * \\
& J_2 B^{\prime -t} J_2^{-1} &
-J_2 D^{\prime -t} B^{\prime t} A^{\prime -t}J_3^{-1}\\
&& J_3 A^{\prime -t} A^{-1} J_3^{-1}
\end{bmatrix}.
$$
In particular, we see that the $\wh j$-involution on
$H^1(\Gal_K, \gr^1\Lie U(A))$
permutes the two direct summands.
\end{proof}

The lemma above immediately implies the following.

\subsection{Corollary}
\label{cor:delta-unitary}
The Galois action $\Gal(K/F)$ on
$H^i(\Gal_K, \gr^j \Lie U(A))$
is a classical $\Delta$-system in the sense of
Definition \ref{def:classical}.

\subsection{Theorem}
\label{thm:main-unitary}
Theorem \ref{thm:classify} holds for $\lsup LG_n=\lsup LU_n$.

\begin{proof}
We apply Theorem \ref{thm:Galois}:
since $K$ is a quadratic extension of $F$, $K\ne \Qp$,
the non-triviality of cup products follows from
Corollary \ref{cor:delta-unitary},
Lemma \ref{lem:unitary-cup-5} and Proposition \ref{prop:unitary-cup-1}.
\end{proof}

\section{B: symplectic groups}~

Since $G=\GSp_{2n}$ is a split group, we have $K=F$.
The Dynkin diagram for $G$
is \dynkin{B}{}.
Thus the maximal proper Levi subgroups of $G$
are of the form
$$
M_k:=\GL_k \times \GSp_{2(n-k)}, {~k\le n}.
$$
Write $P_k$ for the corresponding parabolic subgroup.

Set
$\Omega_k =
\begin{bmatrix}
& I_{n-k} \\
-I_{n-k} &
\end{bmatrix}
$.
We use the following presentation of $\GSp_{2n}$:
$$
\GSp_{2n} = \{X\in \GL_{2n}|
X^t 
\begin{bmatrix}
& & I_k \\
& \Omega_k & \\
-I_k & &
\end{bmatrix}
X = \lambda \begin{bmatrix}
& & I_k \\
& \Omega_k & \\
-I_k & &
\end{bmatrix}
\}
$$

We have 
$$
P_k = \GSp_{2n} \cap
\begin{bmatrix}
\GL_k & \Mat_{k\times (2n-2k)} & \Mat_{k \times k} \\
& \GL_{2n-2k} & \Mat_{(2n-2k)\times k} \\
& & \GL_k
\end{bmatrix}
=:\GSp_{2n} \cap Q_k
$$
where
$Q_k$ is the corresponding parabolic of $\GL_{2n}$.

\subsection{Lemma}
\label{lem:symplectic-1}
If $\bar\rho:\Gal_F \to \GSp_{2n}(\bFp)$
is not elliptic, then there exists a parabolic
$P_k$ through which $\bar\rho$ factors
and $\bar\rho$ is a Heisenberg-type extension of
some $\bar\rho_M:\Gal_F \to M_k(\bFp)$.

\begin{proof}
The proof is similar to that of Lemma \ref{lem:unitary-1}.
A parabolic $\bar\rho$ factors through $P_k$
for some $k$.
Write
$$
\bar\rho = 
\begin{bmatrix}
\bar r_1 & * & * \\& \bar r_2 & * \\& & \bar r_3
\end{bmatrix}
$$
If $\bar r_1$ or $\bar r_3$ is not irreducible,
then $\bar\rho$ also factors through $P_s$
for some $s$ strictly less than $k$.
So we can assume both $\bar r_1$ and $\bar r_3$ are irreducible.
Finally, local Tate duality ensures $\bar\rho$ is a Heisenberg-type extension.
\end{proof}

Write $U$ and $V$ for
the unipotent radical of $Q_k$ and $P_k$,
respectively.
We have
$$
\gr^0 \Lie U = \Mat_{k\times k}
$$
and
$$
\gr^1 \Lie U = \Mat_{k\times 2(n-k)} \times \Mat_{2(n-k)\times k}
$$

Define an $\Delta:=\{1, \j\}$-action on $\Lie U$ by
\begin{eqnarray}
\label{eqn:symp}
\j (x, y) :=& (y^t\Omega_k, \Omega_k x^t),& (x, y)\in \Mat_{k\times 2(n-k)} \times \Mat_{2(n-k)\times k}\\
\j z :=& z^t,& z\in \Mat_{k\times k}.
\end{eqnarray}

\subsection{Lemma}
\label{lem:symp-2}
We have
$\Lie V = (\Lie U)^\Delta$.

\begin{proof}
Clear.
\end{proof}

\subsection{Lemma}
\label{lem:symp-3}
The $\Delta$-action on $\Lie U$
induces a classical $\Delta$-action on
$H^i(\Gal_K, \gr^j\Lie U(A))$
for each $i$, $j$, and $A=\bFp,~\bar\bZ_p$.

\begin{proof}
Clear by Equation (\ref{eqn:symp}).
\end{proof}

\subsection{Corollary}
\label{cor:symp-1}
For Galois representation
$\rho_M
=
\begin{bmatrix}
\tau_a & & \\
& \tau_b & \\
& & \tau_c
\end{bmatrix}
:\Gal_K\to M_k(\bar\bZ_p)$.
The cup product
$$
H^1(\Gal_K, \gr^1\Lie U(\bar\bZ_p))^\Delta\underset{\bar\bZ_p}{\otimes}\bFp
\times
H^1(\Gal_K, \gr^1\Lie U(\bar\bZ_p))^\Delta\underset{\bar\bZ_p}{\otimes}\bFp
\to H^2(\Gal_K, \Mat_{a\times c}(\bFp))^\Delta
$$
is non-trivial.

\begin{proof}
By Corollary \ref{cor:unitary-cup-1}
and Lemma \ref{lem:symp-3},
the cup product is non-trivial unless
$K=\Qp$,
$k = 1$,
and either
$$
\bar\tau_b = \bar\tau_a(-1)^{\oplus 2n-2},~\text{and}~\bar\tau_c = \bar\tau_a(-1)
$$
or
$$
\bar\tau_b = \bar\tau_a^{\oplus 2n-2},~\text{and}~\bar\tau_c = \bar\tau_a(-1).
$$
The symplecticity of $\rho_M$ implies
$$
\bar\tau_a\bar\tau_c = \lambda
$$
and
$$\bar\tau_b^t\Omega_k \bar\tau_b = \lambda \Omega_k$$
where $\lambda$ is the similitude character.
Since $\bar\tau_b=\bar\tau_a(m)I_{2n-2}$ ($m=0,~-1$) is forced to be a scalar matrix,
we have
$$
\bar\tau_a(m)^2=\lambda.
$$
Thus
$$
\bar\tau_a^2(2m) = \bar\tau_a\bar\tau_c = \bar\tau_a(-1)
$$
which implies
$\bFp(1)=\bFp$, which contradicts the fact that $K=\Qp$.
\end{proof}

\subsection{Theorem}
\label{thm:main-symplectic}
Theorem \ref{thm:classify} holds for $\lsup LG_n=\GSp_{2n}$.

\begin{proof}
Combine Theorem \ref{thm:Galois}
and Corollary \ref{cor:symp-1}.
\end{proof}

\section{C: odd and even orthogonal groups}~

Let $G=\GSO_{n}$ be the split orthogonal similitude groups. 
We have $K=F$.
The Dynkin diagram for $G$
is \dynkin{C}{} or \dynkin{D}{}.
Thus the maximal proper Levi subgroups of $G$
are of the form
$$
M_k:=\GL_k \times \GSO_{n-2k},~k\le n.
$$
Write $P_k$ for the corresponding parabolic subgroup.

We use the following presentation of $\GSO_{n}$:
$$
\GSO_{n} = \text{the neutral component of}~ \{X\in \GL_{n}|
X^t 
\begin{bmatrix}
& & I_k \\
& I_{n-2k} & \\
I_k & &
\end{bmatrix}
X = \lambda 
\begin{bmatrix}
& & I_k \\
& I_{n-2k} & \\
I_k & &
\end{bmatrix}
\}
$$

We have 
$$
P_k = \GSO_{n} \cap
\begin{bmatrix}
\GL_k & \Mat_{k\times (n-2k)} & \Mat_{k \times k} \\
& \GL_{n-2k} & \Mat_{(n-2k)\times k} \\
& & \GL_k
\end{bmatrix}
=:\GSO_{n} \cap Q_k
$$
where
$Q_k$ is the corresponding parabolic of $\GL_{n}$.

\subsection{Lemma}
\label{lem:orth-1}
If $\bar\rho:\Gal_F \to \GSO_{n}(\bFp)$
is not elliptic, then there exists a parabolic
$P_k$ through which $\bar\rho$ factors
and $\bar\rho$ is a Heisenberg-type extension of
some $\bar\rho_M:\Gal_F \to M_k(\bFp)$.

\begin{proof}
It is completely similar to Lemma \ref{lem:symplectic-1}.
\end{proof}

Write $U$ and $V$ for
the unipotent radical of $Q_k$ and $P_k$,
respectively.
We have
$$
\gr^0 \Lie U = \Mat_{k\times k}
$$
and
$$
\gr^1 \Lie U = \Mat_{k\times (n-2k)} \times \Mat_{(n-2k)\times k}
$$

Define an $\Delta:=\{1, \j\}$-action on $\Lie U$ by
\begin{eqnarray}
\label{eqn:orth}
\j (x, y) :=& (-y^t, - x^t),& (x, y)\in \Mat_{k\times (n-2k)} \times \Mat_{(n-2k)\times k}\\
\j z :=& -z^t,& z\in \Mat_{k\times k}.
\end{eqnarray}

\subsection{Lemma}
\label{lem:orth-2}
We have
$\Lie V = (\Lie U)^\Delta$.

\begin{proof}
Clear.
\end{proof}

\subsection{Lemma}
\label{lem:orth-3}
The $\Delta$-action on $\Lie U$
induces a classical $\Delta$-action on
$H^i(\Gal_K, \gr^j\Lie U(A))$
for each $i$, $j$, and $A=\bFp,~\bar\bZ_p$.

\begin{proof}
Clear by Equation (\ref{eqn:orth}).
\end{proof}

\subsection{Corollary}
\label{cor:orth-1}
For Galois representation
$\rho_M
=
\begin{bmatrix}
\tau_a & & \\
& \tau_b & \\
& & \tau_c
\end{bmatrix}
:\Gal_K\to M_k(\bar\bZ_p)$.
The cup product
$$
H^1(\Gal_K, \gr^1\Lie U(\bar\bZ_p))^\Delta\underset{\bar\bZ_p}{\otimes}\bFp
\times
H^1(\Gal_K, \gr^1\Lie U(\bar\bZ_p))^\Delta\underset{\bar\bZ_p}{\otimes}\bFp
\to H^2(\Gal_K, \Mat_{a\times c}(\bFp))^\Delta
$$
is non-trivial.

\begin{proof}
By Corollary \ref{cor:unitary-cup-1}
and Lemma \ref{lem:orth-3},
the cup product is non-trivial unless
$K=\Qp$,
$k = 1$,
and either
$$
\bar\tau_b = \bar\tau_a(-1)^{\oplus 2n-2},~\text{and}~\bar\tau_c = \bar\tau_a(-1)
$$
or
$$
\bar\tau_b = \bar\tau_a^{\oplus 2n-2},~\text{and}~\bar\tau_c = \bar\tau_a(-1).
$$
The orthogonality of $\rho_M$ implies
$$
\bar\tau_a\bar\tau_c = \lambda
$$
and
$$\bar\tau_b^t \bar\tau_b = \lambda I_{n-2}$$
where $\lambda$ is the similitude character.
Since $\bar\tau_b=\bar\tau_a(m)I_{n-2}$ ($m=0,~-1$) is forced to be a scalar matrix,
we have
$$
\bar\tau_a(m)^2=\lambda.
$$
Thus
$$
\bar\tau_a^2(2m) = \bar\tau_a\bar\tau_c = \bar\tau_a(-1)
$$
which implies
$\bFp(1)=\bFp$, which contradicts the fact that $K=\Qp$.
\end{proof}

\subsection{Theorem}
\label{thm:main-orthogonal}
Theorem \ref{thm:classify} holds for $\lsup LG_n=\GO_{n}$.

\begin{proof}
Combine Theorem \ref{thm:Galois}
and Corollary \ref{cor:orth-1}.
\end{proof}

\section{Proof of Theorem \ref{thm:main}
and Theorem \ref{thm:EG}}

\begin{proof}[Proof of Theorem \ref{thm:main}]
By Theorem \ref{thm:EG} and Theorem \ref{thm:classify},
it suffices to show Setup \ref{setup}
can always be achieved
by induction on the rank of the group.

If $\bar\rho$ is elliptic, then it is \cite[Theorem C]{L23A}.
If $\bar\rho$ is not elliptic,
then by Theorem \ref{thm:classify} (1),
we can assume
$\bar\rho$
is a Heisenberg-type extension
of $\rho_M$
where $M$ is a maximal proper Levi of $G$.
By the exhaustive description of $M$,
$M$ is a product $\Res_{K/F}\GL_k \times G'$
where $G'$ is a classical group
of smaller rank and the same type
as $G$ itself.
So $\rho_M$ admits a de Rham lift
of regular Hodge type
$\rho_M = (\rho_{\GL_k}, \rho_{G'})$
where $\rho_{\GL_k}:\Gal_K\to \GL_k(\bZp)$
and $\rho_{G'}:\Gal_F\to\lsup LG'(\bZp)$.
If we replace $\rho_{\GL_k}$
by $\chi_{\cyc}^{p^s-1}\rho_{\GL_k}$
where $s$ is a sufficiently large integer,
then $\Ad(\rho_M):\Gal_F \to U^{\ab}(\bQp)$
has Hodge-Tate weights $\ge 2$
and thus $H^1_{\crys}(\Gal_K, U^{\ab}(\bar\bQ_p))
=H^1(\Gal_K, U^{\ad}(\bar\bQ_p))$,
by, for example,
\cite[Lemma 6.3.1]{EG23}.
\end{proof}

\begin{proof}[Proof of Theorem \ref{thm:EG}]
(1)
It follows from \cite[Theorem 5.1.2]{L25A},
whose proof is essential the same as
that of \cite[Theorem 6.1.1, Theorem 6.3.2]{EG23}.

(2)
By \cite[Theorem 4]{L25B},
we have
$$
\dim [\spec R_s/\wh M]
\le \dim \cX_{\lsup LM, \red} - \frac{s^2}{2}
$$
where $\cX_{\lsup LM, \red}$
is the reduced Emerton-Gee stack
and has dimension $[F:\bQ_p]\dim \wh M/B_{\wh M}$
where $B_{\wh M}$ is a Borel of $\wh M$.
By the main theorem of \cite{BG19},
$$
\dim \spec R = (1+\dim \wh M + [F:\bQ_p]\dim \wh M/B_{\wh M}).
$$
and $R$ is $p$-flat, so $\dim R\otimes \bFp = \dim R-1$.
So it remains to show $s \le \lceil s^2/2\rceil$,
which is clear.
\end{proof}

\addcontentsline{toc}{section}{References}
\printbibliography
\end{document}